\documentclass{amsart}

\title[Countable equivariant cellular approximation]{Countable approximation of topological $G$-manifolds: compact Lie groups $G$}

\author[Q Khan]{Qayum Khan}
\address{Department of Mathematics \hfill Saint Louis University \hfill St Louis MO 63103 USA}
\email{khanq@slu.edu}

\usepackage{amsthm,amsmath,amssymb,eucal,mathrsfs}
\usepackage{amscd,pb-diagram,verbatim}
\usepackage[all]{xy}

\usepackage{url}

\usepackage{xcolor}
\definecolor{dark-red}{rgb}{0.4,0.15,0.15}
\definecolor{dark-blue}{rgb}{0.15,0.15,0.4}
\definecolor{medium-blue}{rgb}{0,0,0.5}
\usepackage[colorlinks, linkcolor={dark-red}, citecolor={dark-blue}, urlcolor={medium-blue}]{hyperref} 

\newtheorem{thm}{Theorem}[section]
\newtheorem{conj}[thm]{Conjecture}
\newtheorem{cor}[thm]{Corollary}
\newtheorem{lem}[thm]{Lemma}

\theoremstyle{definition}
\newtheorem{defn}[thm]{Definition}

\newtheorem{rem}[thm]{Remark}
\newtheorem{exm}[thm]{Example}

\numberwithin{equation}{section}

\DeclareMathAlphabet{\matheurm}{U}{eur}{m}{n}

\newcommand{\C}{\mathbb{C}}

\newcommand{\F}{\mathbb{F}}

\newcommand{\N}{\mathbb{N}}

\newcommand{\R}{\mathbb{R}}

\newcommand{\Z}{\mathbb{Z}}

\newcommand{\G}{\Gamma}

\newcommand{\cF}{\mathcal{F}}

\newcommand{\cH}{\mathcal{H}}

\newcommand{\fin}{\matheurm{fin}}

\newcommand{\Homeo}{\mathrm{Homeo}}

\newcommand{\id}{\mathrm{id}}

\newcommand{\Torus}{\mathrm{Torus}}

\newcommand{\bdry}{\partial}

\newcommand{\iso}{\cong}
\newcommand{\longra}{\longrightarrow}

\DeclareMathOperator*{\colim}{colim}
\DeclareMathOperator*{\hocolim}{hocolim}

\newcommand{\ol}[1]{\overline{#1}}

\begin{document}

\begin{abstract}
Let $G$ be a compact Lie group.
(Compact) topological $G$-manifolds have the $G$-homotopy type of (finite-dimensional) countable $G$-CW complexes. 
This partly generalizes Elfving's theorem for locally linear $G$-manifolds \cite{Elfving1}.
\end{abstract}
\maketitle

\section{Topological manifolds}

By \emph{topological manifold}, we shall mean a separable metrizable space $M$ such that each point $x \in M$ has a neighborhood homeomorphic to some $\R^n$.
Notice we allow $n$ to depend on $x$, so that different components of $M$ may have different dimensions.
Note separable and metrizable imply paracompact \cite[4]{Dieudonne} and second-countable.

\begin{lem}\label{lem:top_exhaust}
Let $M$ be a topological manifold.
There exists an increasing sequence $M_0 \subseteq M_1 \subseteq \ldots \subseteq M_i \subseteq \ldots$ of open sets whose union is $M$ such that each closure $
\ol{M_i}$ in $M$ both is compact and is a neighborhood deformation retract in $M_{i+1}$.
\end{lem}

It is a lament of the author to not have discovered a proof from basic topology.

\begin{proof}
If $M$ is compact, then simply use the constant sequence $M_i=M$.
Otherwise, assume $M$ is noncompact.
Since $M$ has only countably many components, say $\{M^c\}_{c=0}^\infty$ and by forming the finite unions $M_i := \bigcup_{c+j=i} M^c_j$ when such sequences $\{M^c_j\}_{j=0}^\infty$ exist, we may further assume that $M$ is connected, say with $n:=\dim\,M$.

First suppose $n \leqslant 4$.
If $n=1$, then $M$ is homeomorphic to either $S^1$ or $\R$ \cite{Gale} hence admits a ($C^\infty$) smooth structure.
If $n=2$, then $M$ admits a triangulation by Rad\'o \cite[2]{Rado} hence a smooth structure by Richards \cite[3]{Richards}\footnote{One can avoid this later classification of Richards and instead use elementary (1920s) methods: the regular neighborhoods of simplices in the second barycentric subdivision are PL handles (see \cite[6.9]{RS} for all $n$); each attaching map of a 1-handle or a 2-handle is isotopic to a smooth one.}.
If $n=3$, then $M$ admits a triangulation by Mo\"{i}se \cite[3]{Moise} hence a smooth structure by Cairns \cite[III]{Cairns}.
If $n=4$, then $M$ admits a smooth structure by Freedman--Quinn \cite[8.2]{FQ}, since $M$ is noncompact connected.
In any case, we may select a smooth structure on $M$.
There exists a Morse function $f: M \longra [0,\infty)$ with each $i \in \N$ a regular value \cite[2.3]{Milnor_Morse} and each $M_i := f^{-1}[0,i)$ precompact \cite[6.7]{Milnor_Morse}.
Furthermore, each $\ol{M_i}$ is a neighborhood deformation retract in $M_{i+1}$ \cite[3.1]{Milnor_Morse}.

Now suppose $n>4$.
The topological manifold $M$ admits a \emph{topological handlebody structure} $\ol{M_0} \subset \ldots \subset \ol{M_i} \subset \ldots$, by Kirby--Siebenmann \cite[III:2.1]{KS} if $n>5$ and by Quinn \cite[9.1]{FQ} if $n=5$.
That is, each $M_i$ is open in $M$, the closure $\ol{M_i}$ is compact, the frontier $\bdry M_i := \ol{M_i}-M_i$ is a bicollared topological submanifold of $M$ (that is, $\ol{M_i}$ is \emph{clean in $M$}), and $\ol{M_{i+1}}$ is the union of $\ol{M_i}$ and a handle.
So, since $\bdry M_i$ is bicollared, each $\ol{M_i}$ is a neighborhood deformation retract in $M_{i+1}$.
\end{proof}

An \emph{absolute neighborhood retract} (with respect to the class of metric spaces, denoted \emph{ANR}) is a metrizable space such that any closed embedding into a metric space admits a retraction of a neighborhood to the embedded image \cite[III:\S6]{Hu}.

A 1951 theorem of O~Hanner has this nonseparable generalization \cite[III:8.3]{Hu}.

\begin{lem}[Hanner]\label{lem:top_Hanner}
Topological manifolds are absolute neighborhood retracts.\end{lem}

\section{Equivariant topological manifolds}

A \emph{$G$-cofibration} is a $G$-map with the $G$-homotopy extension property to any $G$-space.
Some $G$-homotopy commutative squares can be made strictly commutative.

\begin{lem}\label{lem:strict}
Let $G$ be a topological group.
Let $A,B,C,D$ be (topological) $G$-spaces.
Let $i: A \longra B$ be a $G$-cofibration.
Let $f: A \longra C$ and $g: B \longra D$ and $h: C \longra D$ be (continuous) $G$-maps.
Suppose that $g\circ i$ is $G$-homotopic to $h \circ f$.
Then $g$ is equivariantly homotopic to a $G$-map $g': B \longra D$ such that $g' \circ i = h \circ f$.
\end{lem}

\begin{proof}
Write $H: A \times [0,1] \longra D$ for the $G$-homotopy from $g \circ i$ to $h \circ f$.
Since $i: A \longra B$ is a $G$-cofibration, there exists a $G$-homotopy $E: B \times [0,1] \longra D$ from $g$ to some $g'$ such that $E \circ (i \times \id_{[0,1]}) = H$.
Then $g' \circ i = H|(A \times \{1\}) = h \circ f$.
\end{proof}

Recall $G$-cofibrant is equivalent to $G$-NDR for closed $G$-subsets \cite[VII:1.5]{Bredon_GT}.

\begin{defn}\label{def:top}
By a \emph{topological $G$-manifold}, we mean a topological $G$-space $M$ so, for each closed subgroup $H$ of $G$, the \emph{$H$-fixed subspace} is a topological manifold:
\[
M^H ~:=~ \{x \in M \;|\; \forall g \in H : gx=x \}.
\]
\end{defn}

Indeed the restriction of $M^H$ being locally euclidean is not automatic \cite[8]{Bing2}.

\begin{exm}[Bing]\label{exm:Bing}
A $C_2$-action on $\R^4$ exists with fixed set not a $C^0$-manifold.
\end{exm}

\begin{rem}
Let $M$ be a topological manifold with the action of a Lie group $G$.
For any prime $p$, the fixed set $M^H$ is an $\F_p$-cohomology manifold if $H \leqslant G$ is a finite $p$-group (Smith \cite[V:2.2]{Borel_seminar}) or toral group (Connor--Floyd \cite[V:3.2]{Borel_seminar}).
\end{rem}

T~Matumoto introduced the notion of a $G$-CW complex for any $G$ \cite[1.5]{Matumoto_GCW}.

\begin{thm}\label{thm:equi_Lie}
Let $G$ be a compact Lie group.
Any topological $G$-manifold is $G$-homotopy equivalent to a $G$-equivariant countable CW complex.
Furthermore, if the manifold is compact, then the CW complex can be selected to be finite-dimensional.
\end{thm}

\begin{proof}
Let $M$ be a topological $G$-manifold.
By Lemma~\ref{lem:top_Hanner}, each $M^H$ is an ANR.
Also, by the Bredon--Floyd theorem, each compact set in $M$ has only finitely many\footnote{An example with infinitely many orbit types is $M=\N \times S^1$ with $U_1$-action $u(n,z) := (n,u^n z)$.} conjugacy classes of isotropy group \cite[VII:2.2]{Borel_seminar}.
Consider the increasing sequence $\{M_i\}_{i=0}^\infty$ of Lemma~\ref{lem:top_exhaust}.
We may assume that each open set $M_i$ is $G$-invariant by replacement with its $G$-saturation $G M_i$; indeed, the compactness of $G$ implies the compactness of $\ol{M_i}$ \cite[I:3.1(3)]{Bredon_TG}.
Then each $\ol{M_i}$, hence $M_i$, has finitely many orbit types.
Also, the open set $M_i^H = M_i \cap M^H$ in $M^H$ is an ANR by Hanner \cite[III:7.9]{Hu}.
So, since $M_i$ is a separable metrizable finite-dimensional locally compact space, by a criterion of Jaworowski \cite[2.1]{Jaworowski1}\footnote{Beware the $G$-Wojdys{\l}awski theorem therein is incorrect and must be amended by \cite{Jaworowski3}.}, we obtain that $M_i$ is a \emph{$G$-ENR}.
That is, there exist a closed $G$-embedding of $M_i$ in a euclidean $G$-space (equipped with a smooth orthogonal representation of $G$), an open $G$-neighborhood $U$ of $M_i$, and a $G$-retraction $r: U \longra M_i$.
By a theorem of Illman \cite[7.2]{Illman_compactLie}, the smooth $G$-manifold $U$ admits a $G$-CW structure, finite-dimensional and countable.

Write $e: M_i \longra U$ for the $G$-inclusion.
By the $G$-version of Mather's trick \cite{Mather}, the mapping torus $\Torus(r \circ e)$ is $G$-homotopy equivalent to $\Torus(e \circ r)$.
Note $\Torus(r \circ e)$ is $G$-homeomorphic to $M_i \times S^1$ with trivial $G$-action on $S^1$.
By $G$-cellular approximation \cite[II:2.1]{tomDieck}, the $G$-map $e \circ r: U \longra U$ is $G$-homotopic to a cellular $G$-map $c: U \longra U$.
Then $\Torus(e \circ r)$ is $G$-homotopy equivalent to the finite-dimensional countable $G$-CW complex $\Torus(c)$.
Thus $M_i \times S^1$, hence the infinite cyclic cover $M_i \times \R \simeq M_i$, is $G$-homotopy equivalent to a finite-dimensional countable $G$-CW complex $K_i$, namely the bi-infinite mapping telescope of $c$.

Thus, for each $i$, we obtain a $G$-homotopy equivalence $f_i: M_i \longra K_i$.
Choose a $G$-homotopy inverse $\ol{f_i}: K_i \longra M_i$.
Write $c_i: M_i \longra M_{i+1}$ for the $G$-cofibrant inclusion.
By $G$-cellular approximation \cite[II:2.1]{tomDieck}, the composite $f_{i+1} \circ c_i \circ \ol{f_i}$ is $G$-homotopic to cellular $G$-map $g_i : K_i \longra K_{i+1}$.
Recursively by Lemma~\ref{lem:strict} for each $i$, we may reselect $f_{i+1}$ up to $G$-homotopy such that $f_{i+1} \circ c_i = g_i \circ f_i$ holds.
These assemble to a $G$-homotopy equivalence of the mapping telescopes
\[
f: \hocolim_{i \in \N} (M_i,c_i) \longra K := \hocolim_{i \in \N} (K_i, g_i).
\]
Thus, as the $c_i$ are cofibrations, the topological $G$-manifold $M = \colim_i (M_i, c_i)$ is $G$-homotopy equivalent to the countable $G$-CW complex $K$ (\cite[Appendix]{Milnor_Morse}).
\end{proof}


Notice that, since the dimension of the euclidean space equivariantly containing $M_i$ may be unbounded over all $i$, the CW complex $K$ may not be finite-dimensional.

\section{Examples and related results}

\begin{exm}[Bing]\label{exm:horned}
Consider the double $D := E \cup_A E$ of the closed exterior $E$ in $S^3$ of the Alexander horned sphere $A \approx S^2$.
This double has an obvious involution $r_B$ that interchanges the two pieces and leaves the horned sphere fixed pointwise.
R\,H~Bing showed that $D$ is homeomorphic to the 3-sphere \cite{Bing}.
By \ref{thm:equi_Lie}, $(S^3,r_B)$ has the $C_2$-homotopy type of a finite-dimensional countable $C_2$-CW complex.
\end{exm}

\begin{exm}[Montgomery--Zippin]\label{exm:MZ}
Adaptation of Bing's ideas produces an involution $r_{MZ}$ of $S^3$ with fixed set a wildly embedded circle \cite[\S2]{MZ}.
So $(S^3,r_{MZ})$ also has the $C_2$-homotopy type of a finite-dimensional countable $C_2$-CW complex.
\end{exm}

\begin{exm}[Lininger]\label{exm:Lininger}
For each $k \geqslant 3$, there exist uncountably many inequivalent \emph{free} $U_1$-actions on $S^{2k-1}$ with quotient not a $C^0$-manifold \cite[Remark 2]{Lininger}.
Each has the $U_1$-homotopy type of a finite-dimensional countable $U_1$-CW complex.
Indeed, using \cite{AC} for mutation\footnote{His first step is $C^0$-existence of a product tube around a principal orbit \cite[Remark V:4]{MZ_book}.} of $S^{2k-1} \longra \C P_k$ keeps isotropy groups trivial.
\end{exm}

\begin{rem}[Kwasik]\label{rem:continuity}
The case of $G$ finite in Theorem~\ref{thm:equi_Lie} was proven by S~Kwasik \cite[3.6]{Kwasik}.
However, in his first step, he implicitly assumed the continuity of the $G$-action, defined by $g \cdot f := (x \mapsto f(g^{-1} x))$, on the Banach space $B(X)$ of bounded continuous functions $X \longra \R$ equipped with the sup-norm.
This popular assumption was implicit in the infamous assertion of an equivariant version of Wojdys{\l}awski's theorem.
Non-equivariantly it states that the \emph{Kuratowski embedding,} defined by $x \longmapsto (y \mapsto d(x,y))$, of a metrizable space $X$ into $B(X)$ \cite[6]{Kuratowski}, has image closed in its convex hull, where $d$ is any bounded metric on $X$ \cite[7]{Wojdyslawski}.
Nonetheless, Kwasik's assumption is indeed true and proven later by S~Antonyan; it follows from \cite[Proposition 8]{Antonyan1} that it holds when $G$ is a discrete group.
Continuity of the $G$-action on $B(X)$ fails if $G = U_1$ \cite[\S8:I]{Antonyan1}, a flaw in \cite{Kwasik_BADPROOF}.
\end{rem}

\begin{cor}\label{cor:equi}
Let $\G$ be a virtually torsionfree\footnote{Virtually torsionfree holds necessarily for all finitely generated $\G < GL_n(\C)$ by Selberg \cite[8]{Selberg}.  Deligne gives an example of this property failing for a lattice $\G$ in a \emph{nonlinear} Lie group \cite{Deligne}. I conjecture that this corollary is true if ``virtually torsionfree'' is replaced with ``residually finite.''}, discrete group.
Any topological $\G$-manifold with proper action has the $\G$-homotopy type of a countable $\G$-CW complex.
Moreover if the action is cocompact, then the CW complex can be finite-dimensional.
\end{cor}

\begin{proof}
Using intersection with finitely many conjugates, there is a normal, torsionfree subgroup $\G^+$ of $\G$.
Write $G := \G/\G^+$ for the finite quotient group.
Since the $\G$-action on the given topological $\G$-manifold $M^+$ is properly discontinuous and $\G^+$ is torsion-free, $M^+$ is a regular $\G^+$-covering space of the orbit space $M := M^+/\G^+$.
Thus, by the evenly covered property, $M$ is a topological $G$-manifold.
Therefore, by Theorem~\ref{thm:equi_Lie} (or just \cite[3.6]{Kwasik} amended by Remark~\ref{rem:continuity}), there is a $G$-homotopy equivalence $f: X \longra M$ from a $G$-equivariant countable CW complex $X$.

Write $X^+ := f^* M^+$ and consider the pullback diagram of $\G$-spaces and $\G$-maps:
\begin{equation}\label{eqn:pullback}
\begin{CD}
X^+ @>f^+>> M^+\\
@Vf^*qVV @VqVV\\
X @>f>> M.
\end{CD}
\end{equation}
Since $q$ is a regular $\G^+$-covering map, the induced map $f^*q$ given by inclusion-projection is also a regular $\G^+$-covering map.
Then the covering space $X^+$ is a CW complex \cite[IV:8.10]{Bredon_GT}.
This $\G$-invariant CW structure on $X^+$ is a countable $\G$-CW complex \cite[II:1.15]{tomDieck}.
The induced map $f^+$ is a $\G$-homotopy equivalence.
\end{proof}

An application of Corollary~\ref{cor:equi}: any such topological $\G$-manifold, whose fixed set of any finite subgroup is contractible, is a model for the classifying space $E_{\fin}\G$; the $\G$-CW condition is assumed for the existence of classifying maps.
This corollary thus simplifies the fundamentally complicated proofs of \cite[2.3]{CDK1} \cite[2.5]{CDK2}.

\begin{rem}[Freedman, Quinn]
Let $G$ be a topological group.
A \emph{locally linear $G$-manifold} \cite[\S IV:1]{Bredon_TG} (`locally smooth' therein) is a proper $G$-space $M$ such that, for each $x \in M$, there are a $G_x$-invariant subset $x \in S \subset M$ and a homeomorphism $S \longra \R^k$ equivariant with respect to an $O(k)$-representation of the \emph{isotropy group}
\[
G_x ~:=~ \{ g \in G \;|\; gx=x \}
\]
such that the multiplication map $G \times_{G_x} S \longra M$ is an open embedding.
Any smooth $G$-manifold is a locally linear $G$-manifold \cite[2.2.2]{Palais}.
M~Freedman constructed a nonsmoothable free involution $r_F$ on the 4-sphere (more in \cite{KL}).
Hence $(S^4,r_F)$ is a locally linear $C_2$-manifold that is not a smooth $C_2$-manifold.
F~Quinn constructed a locally linear $G$-disc $\Delta$ where $G=C_{15} \rtimes_2 C_4$ such that $\Delta$ does not contain a \emph{finite} $G$-CW complex in its $G$-homotopy type \cite[2.1.4]{Quinn_EndsII}.
So for non-smooth actions there is no $G$-analog of the Borsuk Conjecture \cite[5.3]{West}.
\end{rem}

\begin{rem}[Elfving]
Let $G$ be any Lie group, and further assume that all actions are \emph{proper} in the more specialized sense of Palais \cite[1.2.2]{Palais}.
E~Elfving concluded that any locally linear $G$-manifold has the $G$-homotopy type of a $G$-CW complex \cite{Elfving2}.
To contrast with Theorem~\ref{thm:equi_Lie}, the input type of group is broader, but the input type of manifold is restricted, and the output type of CW complex less sharp.
Philosophically, for a homotopy-type result, one does not need the homeomorphism structure of locally linear tubes, but simply equivariant neighborhood retractions.
\end{rem}

\section{Application of the Hilbert--Smith Conjecture}\label{sec:HilbertSmith}

An early attempt on Hilbert's fifth problem (1900) is this approximation \cite{vonNeumann}.

\begin{lem}[von Neumann]\label{lem:tower}
Any compact Hausdorff group is isomorphic to the filtered limit of an infinite tower of compact Lie groups and smooth epimorphisms.
\end{lem}

\begin{proof}[Proof (Pontrjagin--Weil)]
We outline a later proof, as it is a beautiful interaction of algebra, analysis, and topology \cite[Theorem 54]{Pontrjagin}. 
Haar (1933) showed that any locally compact Hausdorff group $G$ admits a translation-invariant measure $\mu$ on its Borel $\sigma$-algebra.
Since $G$ is compact, we take $\mu$ to be a probability measure.
Urysohn's lemma (1925) implies $G$ is completely Hausdorff.
So on the Hilbert space $\cH := L^2(G,\mu;\C)$ the $G$-action via pre-multiplication is faithful.
Peter--Weyl (1927) showed the sum of all finite-dimensional $G$-invariant $\C$-linear subspaces is dense in $\cH$.
Thus each nonidentity element of $G$ acts nontrivially on such a subspace of $\cH$.

So, since $G$ is compact, it follows that each neighborhood $B$ of the identity $e$ in $G$ contains the kernel $N$ of a finite product of (finite-dimensional) unitary representations.
That is, $B \supset N \triangleleft G$ and $G/N$ is monomorphic to a finite product of unitary groups.
So $G/N$ is a Lie group (von Neumann, 1929).
If $G$ is first-countable, then $e$ has a countable decreasing basis $\{B_i\}_{i=0}^\infty$, so we can also assume decreasing for the corresponding sequence $\{N_i\}_{i=0}^\infty$ of closed normal subgroups of $G$ with $G/N_i$ a (compact) Lie group.
So, since $\bigcap_{i=0}^\infty N_i$ is trivial, the homomorphism $G \longra \lim_i (G/N_i)$ is injective.
It is surjective, since for any $\{a_i N_i \}_{i=0}^\infty$ in the limit, the intersection of the descending chain of the compact sets $a_i N_i \subset G$ is nonempty.

Weil removed first-countability by limiting over a larger directed set \cite[\S25]{Weil}.
Consider the set $\cF$ of finite subsets of $G-\{e\}$ partially ordered by inclusion; note it is directed as $I,J \in \cF$ implies $I \cup J \in \cF$.
For each $e \neq x \in G$, our unitary representation $M_x: G \longra U_{d(x)}$ satisfies $M_x(x)\neq \id$; call the kernel $N_x \triangleleft G$.
For each $I \in \cF$, define $M_I := \prod_{x \in I} M_x$ with kernel $N_I = \bigcap_{x \in I} N_x$, so that again $G/N_I \iso M_I(G) \leqslant \prod_{x \in I} U_{d(x)}$ is a compact Lie group.
Thus $I \longmapsto G/N_I$ is a functor from $\cF^{\mathrm{op}}$ to the category of compact Lie groups and smooth epimorphisms.
Since each $e \neq x \notin N_x$, the homomorphism $q: G \longra \lim_{I \in \cF} (G/N_I)$ is injective.
Since $\{a_I N_I\}_{I \in \cF}$ has finite-intersection property, $q$ is surjective \cite[26.9]{Munkres}.
\end{proof}

I was unable to apply this tower to generalize Theorem~\ref{thm:equi_Lie} to compact Hausdorff $G$.
However, under certain circumstances, $G$ becomes Lie, so that Theorem~\ref{thm:equi_Lie} applies.
If $G$ is compact, by Lemma~\ref{lem:tower}, $G$ \emph{has no small subgroups} if and only if it is Lie.
Notice, for any prime $p$, the additive group $\Z_p$ of $p$-adic integers is compact Hausdorff with arbitrarily small subgroups.
So $G$ is Lie implies it contains no $\Z_p$.

\begin{conj}[Hilbert--Smith]\label{conj:HilbertSmith}
Let $G$ be a locally compact group.
If $G$ admits a faithful action on a connected topological manifold $M$, then $G$ must be a Lie group.
In short, $G$ is Lie if $G \leqslant \Homeo(M)$.
Hilbert's fifth problem is the case $M=G$.
\end{conj}

Partial affirmations exist if one assumes some sorts of regularity of the action.

\begin{thm}[Kuranishi {\cite[4]{Kuranishi} \cite[V:2.2]{MZ_book}}]
Conjecture~\ref{conj:HilbertSmith} is true for $G$, if there exists a $C^1$-manifold structure on $M$ and the action is by $C^1$-diffeomorphisms.
\end{thm}

If $G$ is compact, it is unworthy to use Theorem~\ref{thm:equi_Lie} as the $C^1$-action becomes $C^\omega$ \cite[1.3]{MS}.
Any $C^1$-map (continuous first partial derivatives) is locally Lipschitz.

\begin{thm}[Repov\v{s}--\v{S}\v{c}epin \cite{RvS}]\label{thm:RvS}
Conjecture~\ref{conj:HilbertSmith} is true for $G$, if the action is by Lipschitz homeomorphisms with respect to a riemannian metric on smooth $M$.
\end{thm}

Their argument shows the impossibility of $G$ containing $\Z_p$ for any prime $p$, by otherwise establishing inequalities with Hausdorff dimension for various metrics.
The riemannian hypothesis is only used to obtain equality with covering dimension.

Later it was noticed that the Baire-category part of the argument could be done locally in a euclidean chart, averaging a metric on it over a small subgroup of $\Z_p$.
In this situation, the new lemma is: $M$ and $M/\Z_p$ have equal covering dimensions.

Both short proofs rely on C~T~Yang's 1960 result: $\dim_\Z(M/\Z_p)=\dim_\Z(M)+2$.

\begin{thm}[George-Michael \cite{GeorgeMichael}]\label{thm:GM}
Conjecture~\ref{conj:HilbertSmith} is true for $G$, if the action is by locally Lipschitz homeomorphisms on the connected topological manifold $M$.
\end{thm}

Notice any self-homeomorphism is Lipschitz if the metric space $M$ is compact.
The following consequence is pleasant because of no smoothness on $G$ or the action.

\begin{cor}
Let $G$ be a compact group.
Let $M$ be a topological $G$-manifold.
Suppose the action is faithful and that $M$ is compact and connected.
Then $M$ has the $G$-homotopy type of a $G$-equivariant finite-dimensional countable CW complex.
\end{cor}

To such $(M,G)$, including Examples~\ref{exm:horned}--\ref{exm:Lininger} or others with $C^0$ dynamics, one can now apply equivariant homology, with variable coefficients in abelian groups (Bredon--Matumoto \cite{Matumoto_ECT}) or more generally in spectra (Davis--L\"uck \cite{DL1}).
Therefore, inductively calculable invariants of these wild objects are now available.

\subsection*{Acknowledgements}

Frank Connolly (U Notre Dame) provided early motivation (\ref{cor:equi}).
Craig Guilbault (U Wisconsin--Milwaukee) reminded me topological manifolds have handles (\ref{lem:top_exhaust}).
The following brief but fruitful discussions happened in June 2017, at Wolfgang L\"uck's 60th birthday conference (U M\"unster).
Wolfgang Steimle (U Augsburg) assured me that strictification of homotopy commutative diagrams is easy for squares (\ref{lem:strict}).
Elmar Vogt (Freie U Berlin) excited me further into classical topology (ANRs and Lemma~\ref{lem:tower}).
Finally, individually Jim Davis (Indiana U), Nigel Higson (Penn State U), Wolfgang L\"uck (U Bonn), and Shmuel Weinberger (U Chicago) encouraged me toward the Hilbert--Smith Conjecture (\ref{conj:HilbertSmith}).

The referee corrected some exposition and recommended Examples~\ref{exm:Bing} \& \ref{exm:Lininger}.
Chris Connell (Indiana U) was helpful in discussing Hausdorff dimension (\ref{thm:RvS}, \ref{thm:GM}).

\bibliographystyle{alpha}
\bibliography{CountableApproximation_compactG}

\end{document}